\documentclass[10pt,twoside,leqno]{amsart}
\usepackage{amsmath,amsfonts,amssymb,amsthm}

\theoremstyle{plain}
\newtheorem{teo}{Theorem}[section]
\newtheorem*{teo*}{Theorem}
\newtheorem{co}[teo]{Corollary}
\newtheorem{lemma}[teo]{Lemma}
\newtheorem{prop}[teo]{Proposition}
\newtheorem*{sublema*}{Sublemma}

\theoremstyle{definition}

\renewcommand{\b}{\beta}
\theoremstyle{remark}

\newtheorem*{remarkthep}{Remarks about the proof of Theorems \ref{recthm1} and \ref{legerperm}}
\newtheorem{remark1}{Remark}
\newcommand{\R}{{\mathbb R}} 
\newcommand{\Ha}{{\mathcal H}^{\alpha}}
\newcommand{\ha}{{\mathcal H}}

\newcommand{\Rn}{{\mathbb R}^d}
\newcommand{\Rd}{{\mathbb R}^2}

\newcommand{\N}{\mathbb{N}}
\newcommand{\Z}{\mathbb{Z}}
\newcommand{\C}{\mathbb{C}}

\newcommand{\cc}{{\mathcal C}}

\renewcommand{\a}{\alpha}

\newcommand{\ra}{\rightarrow}

\newcommand{\CC}{{\mathcal C}}

\newcommand{\diam}{{\rm diam}}

\newcommand{\mang}{{\measuredangle}}

\newcommand{\pia}{{p^i_{\alpha,n}}}
\renewcommand{\l}{{\lambda}}

\newcommand{\stm}{\setminus}

\newcommand{\ga}{\gamma_{\alpha}^n}
\newcommand{\al}{\alpha}

\title[CZ kernels, rectifiability and Wolff capacities]{Some Calder\'on-Zygmund kernels and their relations to Wolff capacities and rectifiability}

\author{Vasilis Chousionis}
\address{Vasilis Chousionis. Department of Mathematics and Statistics,
P.O. Box 68,  FI-00014 University of Helsinki, Finland}
\email{vasileios.chousionis@helsinki.fi}
\thanks{V.C. is funded by the Academy of Finland Grant SA 267047. Also, partially supported by the ERC Advanced Grant 320501, while visiting Universitat Aut\`onoma de Bar\-ce\-lo\-na}

\author{Laura Prat}
\address{Laura Prat.  Departament de
Ma\-te\-m\`a\-ti\-ques, Universitat Aut\`onoma de Bar\-ce\-lo\-na, Catalonia}
\email{laurapb@mat.uab.cat}
\thanks{L.P is supported by grants 2009SGR-000420 (Generalitat de Catalunya) and MTM2010-15657 (Spain).}

\begin{document}

\begin{abstract}
We consider the Calder\'on-Zygmund kernels $K_ {\a,n}(x)=(x_i^{2n-1}/|x|^{2n-1+\a})_{i=1}^d$ in $\Rn$ for $0<\al\leq 1$ and $n\in\N$.
We show that, on the plane, for $0<\al<1$, the capacity associated to the kernels $K_{\a,n}$ is comparable to the Riesz capacity 
$C_{\frac23(2-\al),\frac 3 2}$ of non-linear potential theory.  As consequences we deduce the semiadditivity and bilipschitz invariance of this capacity.
Furthermore we show that for any Borel set 
$E\subset\Rn$ with finite length the $L^2(\mathcal{H}^1 \lfloor E)$-boundedness of the singular integral associated to $K_{1,n}$
implies the rectifiability of the set $E$. We thus extend to any ambient dimension, results previously known only in the plane.
\end{abstract}
\maketitle

\section{Introduction and statement of the results}
In this paper we continue the program started in \cite{cmpt} and \cite{cmpt2} where an extensive study of the kernels $x_i^{2n-1}/|x|^{2n}, n \in \N,$ was performed in the plane. We explore the kernels  $K_{\a,n}(x)=(K_{\a,n}^i(x))_{i=1}^d$ in $\Rn$, where
$$K_{\a,n}^i(x)= \frac{x_i^{2n-1}}{|x|^{2n-1+\al}},$$ 
for $0<\a\leq 1$, $n\in\N$, in connection to rectifiability and their corresponding capacities.

For compact sets $E\subset\Rd$, we define 
\begin{equation}\label{capalfa}
\ga(E)=\sup|\langle T,1\rangle|,
\end{equation}
the supremum taken over those real distributions $T$ supported on $E$ such
that for $i=1,2$, the potentials $K_{\a,n}^i*T$
are in the unit ball of $L^\infty(\Rd)$. For $n=1$ and $\al=1$ the capacity $\gamma_1^1$ coincides with the analytic capacity, modulo multiplicative constants 
(see \cite{semiad}) and it is worth mentioning that for $\al=1$ and $n\in\N$ it was proved in \cite{cmpt2} that $\gamma_1^n$ is comparable to analytic capacity. 
Recall that the analytic capacity of a compact subset of the plane is defined by $$\gamma(E)=\sup|f'(\infty)|,$$the supremum taken over the analytic functions on 
$\C\setminus E$ such that $|f(z)|\le 1$ for $z\in\C\setminus E$. Analytic capacity may be written as \eqref{capalfa} interchanging the real by complex distributions 
and the vectorial kernel $K_{\a,n}$ 
by the Cauchy kernel. Therefore, our set function $\ga$ can be viewed as a real variable version of analytic capacity associated to the vector-valued kernel $K_{\a,n}$.

There are several papers where similar capacities have been studied;
in $\Rn$, for $0<\al<1$, it was discovered in \cite{imrn} that compact sets with finite $\al-$dimensional Hausdorff measure have zero $\gamma_\al^1$ capacity 
(for the case of non-integer $\al>1$ one has to assume some extra regularity assumptions on the set, see \cite{imrn} and \cite{illinois}). 
This is in strong contrast with the situation where $\al\in\Z$ (in this case $\al-$dimensional smooth hypersurfaces have positive $\gamma_\al^1$ capacity, see \cite{mp}, 
where they showed that if $E$ lies on a Lipschitz graph, then $\gamma_{d-1}^1(E)$ is comparable to the $(d-1)-$Hausdorff measure ${\mathcal H}^{d-1}(E)$). In \cite{tams} 
the semiadditivity of the $\gamma_\a^1$ was proven for $0<\a<d$ in $\Rn$. 

For $s>0$, $1 < p < \infty$ and $0 < sp\le 2$, the Riesz capacity $C_{s,p}$ of a compact set $K\subset\Rd$, is defined as

\begin{equation}\label{wolffcap}
C_{s,p}(K)=\sup_{\mu}\mu(K)^p,
\end{equation}
where the supremum runs over all positive measures $\mu$ supported on $K$ such that 

$$I_s(\mu)(x)=\int\frac{d\mu(x)}{|x-y|^{2-s}}$$
satisfies $\|I_s(\mu)\|_q\le 1$,  where as usual $q= p/(p-1)$. The capacity $C_{s,p}$ plays a central role in understanding the nature of Sobolev spaces 
(see \cite{adamshedberg} Chapter 1, p. 38).

In \cite{mpv} it was surprisingly shown that in $\Rn$ for $0<\a<1$, 
the capacities $\gamma_\a^1$ and $C_{\frac 2 3(d-\al),\frac 3 2}$ are comparable. 
In this paper we extend the main result from \cite{mpv} on the plane by establishing the equivalence between $\ga$, $0<\al<1$, $n\in\N$ and the capacity $C_{\frac 2 3(2-\al),\frac 3 2}$ of non-linear potential theory.  
Our first main result reads as follows:

\begin{teo}\label{main}
For each compact set $K\subset\Rd$, $0 <\al< 1$ and $n\in\mathbb{N}$ we have
\begin{equation*}
c^{-1}\ C_{\frac 2 3(2-\al),\frac 3 2}(E)\le\ga(E)\le c\ C_{\frac 2 3(2-\al),\frac 3 2}(E).
\end{equation*}
where $c$ is a positive constant depending only on $\al$ and $n$.
\end{teo}

On the plane and for $\a\in (1,2)$ the equivalence of the above capacities is not known. In \cite{env} it was shown that, in $\Rn$, for $0<\a<d$ and $n=1$, 
the first inequality in Theorem \ref{main} holds (replacing $C_{\frac 2 3(2-\al),\frac 3 2}$ by $C_{\frac 2 3(d-\al),\frac 3 2}$). The question concerning the validity
of the inequality $\ga(E)\lesssim C_{\frac 2 3(d-\al),\frac 3 2}(E)$ for all non integer $\a\in (0,d)$ and $n\in\N$ remains open.\newline

Theorem \ref{main} has some interesting consequences. As it is well known, sets with positive capacity $C_{s,p}$ have non finite Hausdorff 
measure ${\mathcal H}^{2-sp}$. Therefore, the same applies to $\ga$, for $0<\a<1$ and $n\in\N$. 
Hence as a direct corollary of Theorem \ref{main} one can assert that $\ga$ vanishes on sets with finite $\Ha$ measure.  
On the other hand, since $C_{s,p}$ is a subadditive set function (see \cite{adamshedberg}, p. 26), $\ga$ is semiadditive, which means that given compact sets $E_1$ and $E_2$
$$\ga(E_1\cup E_2)\le C(\ga(E_1)+\ga(E_2)),$$for some constant $C$ depending on $\al$ and $n$. In fact $\ga$ is countably semiadditive.

Another consequence of Theorem \ref{main} is the bilipschitz invariance of $\ga$, meaning that for bilipschitz homeomorphisms of $\Rd$, $\phi:\Rd\to\Rd$, 
namely $$L^{-1}|x-y|\le|\phi(x)-\phi(y)|\le L|x-y|\,\;\;\;x,\;y\in\Rd$$ one has $$\ga(E)\approx\ga(\phi(E)).$$
The fact that analytic capacity is bilipschitz invariant was a very deep result in \cite{bilipschitz}, see also \cite{gv} and \cite{gpt}.

All advances concerning analytic capacity in the last 40 years, \cite{mmv}, \cite{david}, \cite{semiad}, \cite{bilipschitz}, go through the deep geometric study of the Cauchy transform which was initiated by Calderon in \cite{ca}. In particular it was of great significance to understand what type of geometric regularity does the $L^2(\mu)$-boundedness of the Cauchy transform impose on the underlying measure $\mu$. From the results of David, Jones, Semmes and others, soon it became clear that rectifiability plays an important role in the understanding of the aforementioned problem.  Recall that $n$-rectifiable  sets in $\Rn$ are contained, up to an $\ha^n$-negligible set, in a countable union  of $n$-dimensional Lipschitz graphs.  Mattila, Melnikov and Verdera in \cite{MMV} proved that whenever $E$ is an $1$-Ahlfors-David regular set that the $L^2(\ha^1 \lfloor E)$-boundedness of the Cauchy transform is equivalent to $E$ being $1$-uniformly rectifiable. A set $E$ is called $1$-Ahlfors-David regular, or $1$-AD-regular, if there exists some constant $c$ such that
$$c^{-1}r\leq \ha^1(B(x,r) \cap E) \leq c\,r\quad\mbox{for all} x\in E,
0<r\leq\diam(E).$$
Uniform rectifiability is a influential notion of quantitative rectifiability introduced by David and Semmes, \cite{DS1} and \cite{DS2}. In particular a set $E$ is $1$-uniformly rectifiable if it $1$-AD regular and is containted in an $1$-AD regular rectifiable curve. 
Leg\'er in \cite{Leger} proved that if $E$ has positive and finite length and the Cauchy transform is bounded in $L^2(\ha^1 \lfloor E)$ then $E$ is rectifiable. It is a remarkable fact that the proofs of the aforementioned results depend crucially on a special subtle positivity property of the Cauchy kernel related to an old notion of curvature named after Menger. Given three distinct points $z_1,z_2,z_3 \in \C$ their Menger curvature is 
\begin{equation}
\label{meng}
c(z_1,z_2,z_3)=\frac{1}{R(z_1,z_2,z_3)},
\end{equation} 
where $R(z_1,z_2,z_3)$ is the radius of the circle passing through $x,y$ and $z$. Melnikov in \cite{Me} discovered that the Menger curvature is related to the Cauchy kernel  by the formula 
\begin{equation}
\label{mel}
c(z_1,z_2,z_3)^2= \sum_{s \in S_3} \frac{1}{(z_{s_2}-z_{s_1})\overline{(z_{s_3}-z_{s_1})}},
\end{equation}
where $S_3$ is the group of permutations of three elements. It follows immediately that the permutations of the Cauchy kernel are always positive. Further implications of this identity related to the $L^2$-boundedness of the Cauchy transform where illuminated by Melnikov and Verdera in \cite{MeV}. 

While the Cauchy transform is pretty well understood in this context, very few things are know for other kernels. The David-Semmes conjecture, dating from 1991, asks if the $L^2(\mu)$-boundedness of the operators associated with the $n$-dimensional Riesz kernel $x/|x|^{n+1}$, suffices to imply $n$-uniform rectifiabilty. For $n=1$ we are in the case of the Cauchy transform discussed in the previous paragraph. The conjecture has been very recently resolved by Nazarov, Tolsa and Volberg in \cite{NToV} in the codimension 1 case, that is for  $n=d-1$, using a mix of different deep techniques some of them depending crucially on $n$ being $d-1$. The conjecture is open for the intermediate values $n \in(1,d-1)$  
Recently in \cite{cmpt} the kernels $x_i^{2n-1} / |x|^{2n}, \, x\in \Rd, n \in \N,$ were considered and it was proved that the $L^2$-boundedness of the operators associated with any of these kernels implies rectifiability. These are the only known examples of convolution kernels not directly related to the Riesz kernels with this property. In this paper we extend this result to any ambient dimension $d$. 

For $n \in \N$ and $E \subset \Rn$ with finite length we consider the singular integral operator $T_n=(T^i_n)_{i=1}^d$ where formally
\begin{equation}
\label{operator}
T^i_{n}(f)(x)= \int_E K^i_{1,n}(x-y)f(y) d \ha^1(y)
\end{equation} 
and
\begin{equation*}
K^i_{1,n}(x)=\frac{x_i^{2n-1}}{|x|^{2n}}, \quad x =(x_1,\dots,x_d) \in \Rn \stm \{0\}.
\end{equation*}
We extend Theorem 1.2 and Theorem 1.3 from \cite{cmpt} to any dimension $d$. Our result reads as follows.
\newpage
\begin{teo}
\label{recthm1}
Let $E \subset \Rn$ be a Borel set such that $0<\mathcal{H}^1(E)<\infty$, 
\begin{enumerate}
\item if $T_n$ is bounded in $L^2(\mathcal{H}^1 \lfloor E)$, then the set $E$ is rectifiable.
\item if moreover $E$ is 1-$AD$-regular then  $T_n$ is bounded in $L^2(\mathcal{H}^1 \lfloor E)$ if and only if $E$ is 1-uniformly rectifiable.
\end{enumerate}
\end{teo}

The plan of the paper is the following. In section \ref{secsketch} we outline the proof of Theorem \ref{main} which is based on two main 
technical ingredients: positivity of the quantity obtained when symmetrizing the kernel $K_{\a,n}^i$ and the fact that our kernel localizes in the uniform norm. In Section \ref{secperm} we state all the necessary Propositions involving the permutations of the kernels $K_{\a,n}$. Due to their technical nature we have included the proofs of these results in an Appendix.
Section \ref{secloc} is devoted to the proof of the Localization Theorem for our potentials. In section \ref{secoutline} we complete the proof of the 
main theorem showing that $\ga$ is comparable to $C_{\frac 2 3(2-\al),\frac 3 2}$. Finally in section \ref{rect} we elaborate how Theorem \ref{recthm1} follows from our symmetrization results involving the permutations of the kernels $K_{1,n}$ and \cite{cmpt}.

\section{Sketch of the proof of Theorem \ref{main}}\label{secsketch}

Our proof of Theorem \ref{main} rests on two steps:\newline

{\bf First step. } The first step is the analogue of the main result in the paper \cite{semiad}, that is, the equivalence 
between the capacities $\ga$ and $\gamma_{\a,+}^n$, where for compact sets $K\subset\Rd$, $$\gamma_{\a,+}^n(K)=\sup \mu(K),$$ the supremum taken over those positive measures
$\mu$ with support in $K$ whose vector-valued potential $K_{\a,n}*\mu$ lies on the unit ball of $L^{\infty}(\Rd,\Rd)$ (see also \cite{mpv},
\cite{mpv2}, \cite{tams} and \cite{cmpt2} for related results). Clearly, the quantity $\ga$ is larger or equal than $\gamma_{\a,+}^n$. The reverse inequality can be obtained following Tolsa's approach in \cite{semiad}, which is 
based on two main technical points, the first one is the symmetrization of the kernels $K^i_{\a,n}$, $n\in\mathbb{N}$, $0<\a\le 1$, $i=1,2$,
and the second one is a localization result for $K^i_{\alpha,n}$, $i=1,2$.\newline
In section \ref{secperm}, we deal with the symmetrization process for our kernels. We prove, not only the positivity but the explicit description of the quantity 
obtained when symmetrizing the kernels $K^i_{\a, n}$, for $0<\a\le 1$. This will allow us to study the $L^2$-boundedness of the operators with kernel $K^i_{\a, n}$.\newline
The localization result needed in our setting is written in section \ref{secloc}. Specifically, we prove that there exists a positive constant $C$ such that, for each 
compactly supported distribution $T$ and for each coordinate $i$, we have
\begin{equation}\label{prelloc}\left\|K_{\a,n}^i*\varphi_QT\right\|_{\infty}\le C\left\|K_{\a,n}^i*T\right\|_{\infty},\end{equation}
for each square $Q$ and each $\varphi_Q\in {\CC}_0^{\infty}(Q)$ satisfying $\|\varphi_Q\|_\infty\le C$, 
$\|\nabla\varphi_Q\|_\infty\le l(Q)^{-1}$ and $\|\Delta\varphi_Q\|_\infty\le l(Q)^{-2}$, where $l(Q)$ denotes the sidelength of the cube $Q$.
Once the symmetrization and \eqref{prelloc} is at our disposal, Tolsa's machinery applies straighforwardly as was already explained in \cite[Section 2.2]{mpv}.

{\bf Second step. } Once step 1 is performed, i.e. the comparability between the capacities $\ga$ and $\gamma_{\a,+}^n$, we complete the proof 
of the main theorem showing that $\gamma_{\a,+}^n$ is equivalent to $C_{\frac 2 3(2-\a),\frac 3 2}$ in section \ref{secoutline}.

\section{Permutations of the kernels $K_{\alpha,n}$}\label{secperm}
For any three distinct $x,y,z\in\Rn$,  we consider the symmetrization of the kernels $K^i_{\a,n}$:
\begin{equation}
\label{permdef}
\begin{split}
p^i_{\a,n}(x,y,z) &= K^i_{\a,n}(x-y)\,K^i_{\a,n}(x-z) + K^i_{\a,n}(y-x)\,K^i_{\a,n}(y-z) \\&\quad \quad \quad + K^i_{\a,n}(z-x)\,K^i_{\a,n}(z-y).
\end{split}
\end{equation}

We prove that the permutations $p_{\a,n}^i(x,y,z)$, $n\in\N$ and $0<\a<1$ behave like
the inverse of the largest side of the triangle determined by the points $x,y,z$ to the power $2\a$. We also prove a comparability result of the permutations with 
Menger curvature when $\a=1$ which is essential in order to extend the rectifiabilty results from \cite{cmpt}. It is an interesting fact that our proofs also depend on Heron's formula of Euclidean geometry. In order to enhance readability we chose to include the rather lengthy proofs of the following propositions in an Appendix. 



\begin{prop}\label{lemmaperm}
Let $0<\a<1$ and $x,y,z$ be three distinct points in $\Rn$. For $1\le i\le d$ we have
\begin{equation}\label{icomparability}
\frac{A(n, d, \a) \ M_i^{2n}}{L(x,y,z)^{2\a+2n}}\le p^i_{\a,n}(x,y,z)\le \frac{B(n, d, \a)}{L(x,y,z)^{2\a}}
\end{equation}
where $M_i=\max\{|y_i-x_i|,|z_i-y_i|,|z_i-x_i|\}$, $L(x,y,z)$ denotes the length of the largest side of the triangle determined by the three points $x,y,z$ and $A(n, d, \a), B(\a,n)$ 
are some positive constants depending only on $d,\a,n$.
\end{prop}


We also consider 
\begin{equation}
\label{permdef2}
p_{\a,n}(x,y,z)=\sum_{i=1}^d p^i_{\a,n}(x,y,z).
\end{equation}
Proposition \ref{lemmaperm} allows us to prove the following:

\begin{co}
\label{permalfa}
Let $0<\a<1$ and $x,y,z$ be three distinct points in $\Rn$. Then the following holds
\begin{equation*}
\frac{A(n,d, \a)}{L(x,y,z)^{2\a}}\le p_{\a,n}(x,y,z)\le \frac{B(n,d,\a)}{L(x,y,z)^{2\a}}
\end{equation*}
where $L(x,y,z)$ denotes the length of the largest side of the triangle determined by the three points $x,y,z$ and $A(n,d,\a),B(n,d,\a)$ are some 
positive constants depending on $n$, and $n,d,\a$ respectively.
\end{co}

\begin{proof} For $1\le i\le d$ set $M_i=\max\{|y_i-x_i|,|z_i-y_i|,|z_i-x_i|\}$. Without loss of generality assume that 
$M_1=\max\{M_i: 1\le i\le d\}$. Then $M_1$ is comparable to $L(x,y,z)$. The corollary follows from Proposition \ref{lemmaperm}.\end{proof}

For any two distinct points $x_1,x_2 \in \R^d$ we denote by $L_{x_1,x_2}$ the line which contains them and for any two lines $L_1$ and $L_2$ we denote by 
$\mang (L_1,L_2)$ the smallest angle between $L_1$ and $L_2$.
\begin{prop}
\label{posperm}
For any three distinct points $x,y,z \in \R^d$ and $i=1,\dots,d$,
\begin{enumerate}
\item[(i)] $p^i_{1,n}(x,y,z)\geq 0$ and vanishes if and only if $x,y,z$ are collinear or the three points lie on a $(d-1)$-hypersurface perpendicular to the $i$ axis, 
that is $x_i=y_i=z_i$.
\item[(ii)] If $V_j=\{x_j=0\}$ and $\mang{(V_j,L_{x,y})}+\mang{(V_j,L_{x,z})}+ \mang{(V_j,L_{y,z})}\geq \theta_0>0$, then
$$\sum_{i\neq j} p^i_{1,n}(x,y,z)\geq C(\theta_0)c(x,y,z)^2.$$
\end{enumerate}
\end{prop}

\section{Growth conditions and localization}\label{secloc}

\subsection{Growth conditions}

Recall that for a compactly supported
distribution $T$ with bounded Cauchy potential 
\begin{equation*}
\begin{split}
|\langle T,\varphi_Q\rangle |&=\left |\left\langle T,\frac{1}{\pi
z}*\overline\partial\varphi_Q\right\rangle \right|=
\left|\left\langle \frac{1}{\pi z}*T, \overline\partial\varphi_Q
\right\rangle\right|\\*[7pt] &\leq \frac 1 \pi\left\|\frac 1{z}*T
\right\|_\infty \,\|\overline\partial\varphi_Q\|_{L^1(Q)}\leq \,
\frac 1 \pi\left\|\frac 1{z}*T \right\|_\infty \,l(Q),
\end{split}
\end{equation*}
whenever $\varphi_Q$ satisfies
$\|\overline\partial\varphi_Q\|_{L^1(Q)} \le l(Q). $
\vspace{.8cm}

We want to deduce a similar growth condition for the case of having bounded $T*K_{\a,n}^i$, $i=1,2,$ potentials. This is crucial in obtaining the localization 
result in Lemma \ref{localization}. The above written argument is based on the fact that one can recover $f$ using the formula $f=(1/\pi)\overline{\partial}f*1/z$.
Therefore,  we need an analogous reproduction formula for 
the kernels $K_{\a,n}^i$, $i=1,2$. In \cite{imrn} (Lemma 3.1)
a reproduction formula for $x_i/|x|^{1+\al}$, $0<\al<1$, in $\mathbb{R}^d$ was found. In our current setting, the kernels depend on $n \in \N$ hence the arguments are more technically involved.
\begin{lemma}\label{reproductionformula}
If a function $f$ has continuous derivatives up to order two, then it is representable in the form
\begin{equation}\label{repr}
 f(x)=(\varphi_1*K_1)(x)+(\varphi_2*K_2)(x),\;\;x\in\Rd,
\end{equation}
where for $i=1,2$,
$$\varphi_i=S_i(\Delta f)*\frac{x_i}{|x|^{3-\alpha}}:=\left(c\Delta f+\widetilde S_i(\Delta f)\right)*\frac{x_i}{|x|^{3-\alpha}},$$
for some constant $c$ and Calder\'on-Zygmund operators $\widetilde S_1$ and $\widetilde S_2$.

\end{lemma}

Once Lemma \ref{reproductionformula} is available, we obtain the desired growth condition for our compactly supported distribution $T$ with bounded  potentials $K_{\a,n}^i*T$, $i=1,2$:

\begin{equation}\label{ourgrowth}
\begin{split}
|\langle T,\varphi_Q\rangle |&=\left |\left\langle T,K_1*S_1(\Delta\partial_1\varphi_Q)*\frac1{|x|^{1-\al}}+K_2*S_2(\Delta\partial_2\varphi_Q)*\frac1{|x|^{1-\al}}\right\rangle \right|\\*[7pt] &\le
\sum_{i=1}^2\left|\left\langle S_i(K_{\a,n}^i*T),\Delta\partial_i\varphi_Q*\frac1{|x|^{1-\al}}
\right\rangle\right|\\*[7pt] &\leq  \sum_{i=1}^2\left\|S_i(K_{\a,n}^i*T)
\right\|_{\tiny{\mbox{BMO}}} \,\left\|\Delta\partial_i\varphi_Q*\frac1{|x|^{1-\al}}\right\|_{H^1(\Rd)}\\*[7pt] &\leq\,
C \,l(Q)^\al,
\end{split}
\end{equation}
whenever $\varphi_Q$ satisfies the conditions 
\begin{equation}\label{normalization1}
\left\|\Delta\partial_i\varphi_Q*\frac1{|x|^{1-\al}}\right\|_{H^1(\Rd)} \le l(Q)^\al,\,\,\,\mbox{for }\, i=1,2. 
\end{equation}

Observe that the penultimate inequality in \eqref{ourgrowth} comes from the fact that Calder\'on-Zygmund operators send $L^\infty$ to BMO. Recall that a function $f\in\mbox{H}^1(\Rd)$ if and only if
$f\in L^ 1(\Rd)$  and all its Riesz transforms $R_j$, $1\le j\le2,$ (the Calder\'on-Zygmund operators with Fourier multiplier $\xi_j/|\xi|$) are also in $L^1(\Rd)$. The norm of $H^1(\Rd)$ is defined as 
$$\|f\|_{H^1(\Rd)}=\|f\|_{L^1(\Rd)}+\sum_{j=1}^2\|R_j(f)\|_{L^1(\Rd)}.$$

We now formulate  a definition. We say that a distribution $T$ has growth $\alpha$ provided that
$$G_\al(T)=\sup_{\varphi_Q}\frac{|\langle T,\varphi_Q\rangle|}{l(Q)^\al}<\infty,$$
where the supremum is taken over all $\varphi_Q\in \cc_0^\infty(Q)$ satisfying the normalization inequalities \eqref{normalization1} (see also \cite{mpv2} and \cite{tams}, for similar conditions).
The normalization in the $H^1$ norm is the right condition to impose, as will
become clear later on. Recall that a positive Radon measure has growth $\al$ provided $\mu(B(x,r))\le Cr^\al$, for $x\in\Rd$ and $r\ge 0$. 
For positive Radon measures $\mu$ in $\Rd$ the preceding notion
of $\al$ growth is equivalent to the usual one. 

Notice that from \eqref{ourgrowth}, if $T$ is a compactly supported distribution with bounded potentials $K_1*T$ and $K_2*T$, then $T$ has growth $\al$.

For the proof of Lemma \ref{reproductionformula} we need to compute the Fourier transform of the kernels 
$K_{\a,n}^i(x)=x_i^{2n-1}/|x|^{2n-1+\al}$, $1\le i\le 2$, $n\in\N$, $0<\al<1$ (see Lemma 12 in \cite{cmpt2} 
for the case $\al=1$). 
\begin{lemma}\label{fourier}
 For $1\le i\le 2$, $n\in\N$ and $0<\al<1$,
\begin{equation*}
 \widehat{K_{\a,n}^i}(\xi)=c\frac{\xi_i}{|\xi|^{2n+1-\al}}p(\xi_1,\xi_2),
\end{equation*}
for some homogeneous polynomial $p(\xi_1,\xi_2)$ of degree $2n-2$ with no non-vanishing zeros and some positive constant $c:=c(\al,n)$. 
\end{lemma}

To prove Lemma \ref{fourier}, the following identity is vital.

\begin{lemma}\label{coeff} For $n\in\N$ and $0\le l\le n-1$,
$$\sum_{k=1}^{l+1}(-1)^k\frac{(1-\a)(3-\a)\cdots(2k-1-\a)}{2^{n-k}\ (2k-1)!\ (l+1-k)!}=-\frac{(1-\a)\a(\a+2)(\a+4)\dots(\a+2(l-1))}{2^{n-1-l}\ (2l+1)!}.$$
\end{lemma}
\begin{proof}Consider the polynomial
\begin{equation}\label{polynomial1}
p(\a)=\sum_{k=1}^{l+1}(-1)^k\frac{(1-\a)(3-\a)\dots(2k-1-\a)}{ 2^{n-k}\ (2k-1)!\ (l+1-k)!\ }.
\end{equation}
It follows immediately that $\a=1$ is a root of $p$. In the following we will show that $0,-2,-4,\dots,-2(l-1)$ are also roots of $p$. For $j=0,1,\dots,l-1$
\begin{equation*}
\begin{split}
p(-2j)&=\sum_{k=1}^{l+1} (-1)^k\frac{(1+2j)(3+2j)\dots(2k-1+2j)}{(2k-1)!2^{n-k}(l+1-k)!}\\
&=\frac{1}{1\cdot 3\dots (2j-1)}\sum_{k=1}^{l+1} (-1)^k\frac{(2k+2j)!}{(2k-1)!2^{n-k}(l+1-k)!2^{k+j}(k+j)!}\\
&=\frac{1}{1\cdot 3\dots (2j-1)\ 2^{n+j}}\sum_{k=1}^{l+1}(-1)^k \frac{2k (2k+1)\dots (2k+2j)}{(k+j)!(l+1-k)!}\\
&=\frac{1}{1\cdot 3\dots (2j-1)\ 2^{n+j}}\cdot\\
&\quad\quad\quad\sum_{k=1}^{l+1}(-1)^k \frac{2^{j+1} k\cdot (k+1) \dots (k+j) \ (2k+1)(2k+3)\dots (2k+2j-1)}{(k-1)!k \cdot (k+1)\dots (k+j)\ (l+1-k)!}.
\end{split}
\end{equation*}Hence
\begin{equation*}
\begin{split}
p(-2j)&=\frac{1}{1\cdot 3\dots (2j-1)\ 2^{n-1}}\sum_{k=1}^{l+1}(-1)^k\frac{ \ (2k+1)(2k+3)\dots (2k+2j-1)}{(k-1)!\ (l+1-k)!}\\
&=\frac{1}{1\cdot 3\dots (2j-1)\ 2^{n-1}}\sum_{k=1}^{l+1} \sum_{i=0}^j \frac{(-1)^kc_i k^i}{(k-1)!\ (l+1-k)!}\\
&=\frac{1}{1\cdot 3\dots (2j-1)\ 2^{n-1}} \sum_{i=0}^jc_i\sum_{k=1}^{l+1} \frac{ (-1)^kk^i}{(k-1)!\ (l+1-k)!}
\end{split}
\end{equation*} 
Therefore in order to prove that $-2j, j=0,\dots,l-1$ are roots of $p$ it suffices to show that 
\begin{equation*}
\sum_{k=1}^{l+1} \frac{(-1)^k k^i}{(k-1)!\ (l+1-k)!}=-\sum_{m=0}^l \frac{(-1)^m (m+1)^i}{m!\ (l-m)!}=0
\end{equation*}
for $i=0,\dots,j$.  This will follow  immediately if we show that for any $l\geq1$
\begin{equation}
\label{inductionone}
\sum_{m=0}^l \frac{(-1)^m m^i}{m!\ (l-m)!}=0
\end{equation}
for $i=0,\dots,j$.
We will prove \eqref{inductionone} by induction. For $i=0$ we have that for any $l \geq1$
\begin{equation*}
\sum_{m=0}^l \frac{(-1)^m }{m!\ (l-m)!}=\frac{1}{l!}\sum_{m=0}^l \binom{l}{m}(-1)^m=\frac{(1-1)^l}{l!}=0.
\end{equation*}
We will now assume that for any $l\geq 1$
\begin{equation}
\label{indhyp}
\sum_{m=0}^l \frac{(-1)^m m^i}{m!\ (l-m)!}=0
\end{equation}
for $i=0,\dots,j-1$. Then by \eqref{indhyp},
\begin{equation*}
\begin{split}
\sum_{m=0}^l \frac{(-1)^m m^j}{m!\ (l-m)!}&=\sum_{m=1}^l \frac{(-1)^m m^{j-1}}{(m-1)!\ (l-m)!}\\
&=\sum_{N=0}^{l-1}\frac{(-1)^{N+1} (1+N)^{j-1}}{N!\ (l-1-N)!}\\
&=\sum_{N=0}^{l-1}\frac{(-1)^{N+1} \sum_{i=0}^{j-1}\binom{j-1}{i}N^i}{N!\ (l-1-N)!}\\
&=\sum_{i=0}^{j-1}\binom{j-1}{i}\sum_{N=0}^{l-1}\frac{(-1)^{N+1} N^{i}}{N!\ (l-1-N)!}\\
&=0.
\end{split} 
\end{equation*}

Hence we have shown that
\begin{equation*}
p(\a)=C(l)(1-\a)\a(\a+2)\dots(\a+2(l-1)).
\end{equation*}
Plugging this into \eqref{polynomial1} we see that $-C(l)$ is the coefficient of 
the greatest degree monomial of the polynomial in \eqref{polynomial1}, that is,
\begin{equation*}
C(l)= -\text{coefficient of }\a^{l+1} =-\frac{(-1)^{l+1}(-1)^{l+1}} {(2l+1)! \ 2^{n-l-1}} =-\frac{2^{l+1}}{2^n (2l+1)!}.
\end{equation*}
\end{proof}

\begin{proof}[Proof of lemma \ref{fourier}] Without loss of generality fix $i=1$. Notice that for some constant $c$, $$\displaystyle{\widehat{K_{\a,n}^i}(\xi)=c\;\partial_1^{2n-1}|\xi|^{2n-3+\al}.}$$  
To compute $\partial_1^{2n-1}|x|^\beta$, 
for $\b=2n-3+\a$, we will use the following formula from \cite{lz}:

\begin{equation*}
L(\partial)E_n=\sum_{k=0}^{n-1}\frac 1{2^k\;k!}\;\Delta^k L(x)\left(\frac 1 r\frac{\partial}{\partial r}\right)^{2n-1-k}E_n(r),
\end{equation*}
where $r=|x|$ and $L(x)=x_1^{2n-1}$.  First notice that for $0\le k\le n-1$, we have

$$\Delta^k(x_1^{2n-1})=\binom{2n-1}{2k}(2k)!\;x_1^{2n-2k-1},$$
and one can check that
$$\left(\frac 1 r\frac{\partial}{\partial r}\right)^{m}r^\beta=\b(\b-2)\dots(\b-2m+2)r^{\b-2m}.$$
Therefore for $E_n=|x|^\beta$ and $\beta=2n-3+\a$,
\begin{equation*}
\begin{split}
\partial_1^{2n-1}|x|^\beta&=\sum_{k=0}^{n-1} \binom{2n-1}{2k}\frac{(2k)!}{2^k k!}x_1^{2n-2k-1}\b (\b-2)\dots(\b-2(2n-1-k)+2)r^{\b-2(2n-1-k)} \\
&=\frac{x_1}{|x|^{4n-2-\b}}(2n-1)!\b(\b-2)\dots(\b-2(n-2)) \cdot\\
&\quad\quad\quad\sum_{k=0}^{n-1} \frac{x_1^{2(n-k-1)}|x|^{2k}}{(2n-1-2k)!\ 2^k \ k!}(\b-2(2n-2-(n-1)))\dots(\b-2(2n-2-k))\\
&=c(n)\frac{x_1}{|x|^{2n+1-\a}}\sum_{k=0}^{n-1} \frac{x_1^{2(n-k-1)}|x|^{2k}}{(2n-1-2k)!\ 2^k\ k!}(-1+\a)(-3+\a)\dots(-2(n-k)+1+\a).
\end{split}
\end{equation*}

Therefore
\begin{equation}
 \label{polyn}
\partial_1^{2n-1}|x|^\beta =c(n)\frac{x_1}{|x|^{2n+1-\a}}\sum_{k=0}^{n-1} a_k x_1^{2(n-k-1)}|x|^{2k},
\end{equation}
where $$a_k=(-1)^{n-k} \frac{(1-\a)(3-\a)\dots(2(n-k)-1-\a)}{(2n-1-2k)!\ 2^k \ k!}$$

We claim that the homogeneous polynomial of degree $2n-2$ in \eqref{polyn}, namely, 
\begin{equation}
 p(x_1,x_2)=\sum_{k=0}^{n-1} a_k x_1^{2(n-k-1)}|x|^{2k},
\end{equation}
has negative coefficients. Notice that
\begin{equation*}\begin{split}
p(x)&=\sum_{k=0}^{n-1}\;a_k \;x_1^{2(n-k-1)}\;(x_1^2+x_2^2)^{k}=\sum_{k=0}^{n-1}\sum_{j=0}^{k}a_k\binom{k}{j}x_1^{2(n-k+j-1)}x_2^{2(k-j)}\\&=\sum_{l=0}^{n-1}b_{2l}x_1^{2m}x_2^{2(n-1-l)},
\end{split}
\end{equation*}
where for $0\le l\le n-1$, 
$$b_{2l}=\sum_{k=1}^{l+1}a_{n-k}\binom{n-k}{l+1-k}=\sum_{k=1}^{l+1}\frac{(-1)^k(1-\a)(3-\a)\dots(2k-1-\a)}{(2k-1)!\ 2^{n-k}\ (l+1-k)!\ (n-l-1)!}.$$
Applying now Lemma \ref{coeff}, we get that the coefficients $b_{2l}$, $0\le l\le n-1$, of the polynomial $p$ are negative.
\end{proof}

Now we are ready to prove the reproduction formula that will allow as to obtain the localization result that we need.

\begin{proof}[Proof of Lemma \ref{reproductionformula}] By lemma \ref{fourier}, the Fourier transform of \eqref{repr} is 
$$\widehat f(\xi)=\widehat\varphi_1(\xi)\frac{\xi_1}{|\xi|^{3-\a}}\frac{p(\xi_1,\xi_2)}{|\xi|^{2n-2}}+\widehat\varphi_2(\xi)\frac{\xi_2}{|\xi|^{3-\a}}\frac{p(\xi_2,\xi_1)}{|\xi|^{2n-2}},$$
where $p$ is some homogeneous polynomial of degree $2n-2$ with no non-vanishing zeros. 

Define the operators $R_1, R_2$ associated to the kernels 
$\displaystyle{\widehat r_1(\xi_1,\xi_2)=\frac{p(\xi_1,\xi_2)}{|\xi|^{2n-2}}}$ and \newline$\displaystyle{\widehat r_2(\xi_1,\xi_2)=\widehat r_1(\xi_2,\xi_1)}$ respectively.
Since $p$ is a homogeneous polynomial of degree $2n-2$, it can be decomposed as $$p(\xi_1,\xi_2)=\sum_{j=0}^{n-1}p_{2j}(\xi_1,\xi_2)|\xi|^{2n-2-2j},$$ with $p_{2j}$ being 
homogeneous harmonic polynomials of degree $2j$ (see  \cite[Section 3.1.2, p. 69]{stein}). Hence, the operators $R_i$, $i=1,2,$ can be written as
\begin{equation}\label{invertible}
R_if=cf+\mbox{p.v.}\frac{\Omega(x/|x|)}{|x|^2}*f,
\end{equation}
for some constant $c$ and $\Omega\in\mathcal{C}^\infty$ with zero average. Therefore, by  \cite[Theorem 4.15, p. 82]{duandikoetxea} the operators $R_i$, $1\le i\le 2$,
are invertible operators and the inverse operators, say $S_i$ have the same form as $R_i$. This means that the operators $S_i$, $i=1,2,$, with kernels 
$\displaystyle{\widehat s_1(\xi_1,\xi_2)=\frac{|\xi|^{2n-1}}{p(\xi_1,\xi_2)}}$ and $\widehat s_2(\xi_1,\xi_2)=\widehat s_1(\xi_2,\xi_1)$ respectively, 
can be written as in \eqref{invertible}. 
Finally, setting $$\varphi_i=S_i(\Delta f)*\frac{x_i}{|x|^{3-\a}}$$ for $i=1,2,$ finishes the proof of the Lemma. 
\end{proof}

\subsection{Localization}

In what follows, given a square $Q$, $\varphi_Q$ will denote an infinitely differentiable function supported on $Q$ and such that $\|\varphi_Q\|_\infty\le C$, 
$\|\nabla\varphi_Q\|_\infty\le l(Q)^{-1}$ and $\|\Delta\varphi_Q\|_\infty\le l(Q)^{-2}$.

The localization lemma presented in the following is an extension of Lemma 14 in \cite{cmpt2} for $0<\al<1$.

\begin{lemma}\label{localization}
Let $T$ be a compactly supported distribution in $\Rd$ with growth $\al$, $0<\al<1$, such that $(x_i^{2n-1}/|x|^{2n-1+\al})*T \in L^\infty(\Rd)$ for some $n\in\N$ 
and $1\le i\le 2$. Then $(x_i^{2n-1}/|x|^{2n-1+\al})*\varphi_Q\in L^\infty(\Rd)$ and
$$\left\|\frac{x_i^{2n-1}}{|x|^{2n-1+\al}}*\varphi_QT\right\|_\infty\le C\left(\left\|\frac{x_i^{2n-1}}{|x|^{2n-1+\al}}*T\right\|_{\infty}+G_\al(T)\right),$$ for some positive constant $C$. 
\end{lemma}

The next lemma states a sufficient condition for a test function to satisfy the normalization conditions in \eqref{normalization1}.

\begin{lemma}\label{nicecondition} Let $f_Q$ be a test function supported on a square $Q$, satisfying $\|\Delta f_Q\|_{L^1(Q)}\le C.$
Then,
$$\left\|\Delta\partial_if_Q*\frac1{|x|^{1-\al}}\right\|_{H^1(\Rd)}\le l(Q)^\al,\,\,\,\mbox{for }\, i=1,2. $$
\end{lemma}

\begin{proof}[Proof of Lemma \ref{nicecondition}]
We have to show that for $i=1,2,$
\begin{equation}\label{L1}
\left\|\Delta\partial_if_Q*\frac1{|x|^{1-\al}}\right\|_{L^1(\Rd)}\le l(Q)^\al
\end{equation}
and
\begin{equation}\label{L1riesz}
\left\|R_j(\Delta\partial_if_Q)*\frac1{|x|^{1-\al}}\right\|_{L^1(\Rd)}\le l(Q)^\al,\;\;\;j=1,2,
\end{equation}
where $R_j$, $j=1,2$, is the $j$-th component of the Riesz operator with kernel $x_j/|x|^3$.

\begin{equation*}
\begin{split}
\left\|\Delta\partial_if_Q*\frac1{|x|^{1-\al}}\right\|_{L^1(\Rd)}&=\left\|\Delta\partial_if_Q*\frac1{|x|^{1-\al}}\right\|_{L^1(2Q)}+\left\|\Delta\partial_if_Q*\frac1{|x|^{1-\al}}\right\|_{L^1((2Q)^c)}\\*[7pt] &= A+B.
\end{split}
\end{equation*}

We estimate first the term $A$. By taking one derivative from $f$ to the kernel, using Fubini and the fact that $\|\Delta f_Q\|_{L^1}\le C$, we obtain 
\begin{equation*}
\begin{split}
A=\left\|\Delta\partial_if_Q*\frac1{|x|^{1-\al}}\right\|_{L^1(2Q)}&\le\int_{2Q}\int_Q\frac{|\Delta f_Q(x)|}{|x-y|^{2-\al}}dxdy\le Cl(Q)^\al.
\end{split}
\end{equation*}

To estimate term $B$ we bring the Laplacian from $f_Q$ to the kernel $|x|^{\al-1}$ and then use Fubini, the Cauchy-Schwartz inequality and a well 
known inequality of Maz'ya,  \cite[1.1.4, p. 15]{mazya},  and \cite[1.2.2, p. 24]{mazya} , stating that
$\|\nabla f_Q\|_2\le C\|\Delta f_Q\|_1$. Hence,

\begin{equation*}
\begin{split}
\left\|\Delta\partial_if_Q*\frac1{|x|^{1-\al}}\right\|_{L^1((2Q)^c)}&\le C \int_{(2Q)^c}\int_Q\frac{|\partial_i f_Q(x)|}{|x-y|^{3-\al}}dxdy \\*[7pt]&\le C\; \|\nabla f_Q\|_{L^1(Q)}\;l(Q)^{\al-1}\le C\; \|\nabla f_Q\|_2 \;l(Q)^\al\le C\;l(Q)^\al,
\end{split}
\end{equation*}
the last inequality coming from the hypothesis $\|\Delta f_Q\|_{L^1}\le C$.
This finishes the proof of  \eqref{L1}. To prove \eqref{L1riesz}, we remark that, 
\begin{equation}\label{remark}
 \frac{x_j}{|x|^3}*\frac1{|x|^{1-\al}}=c\frac{x_j}{|x|^{2-\al}},
\end{equation} for some constant $c$. This can be seen by computing the Fourier transform of the above kernels. Using this fact, we obtain that

\begin{equation*}
\begin{split}
\left\|R_j(\Delta\partial_if_Q)*\frac1{|x|^{1-\al}}\right\|_{L^1(\Rd)}&=\left\|\Delta\partial_if_Q*\frac{x_j}{|x|^3}*\frac1{|x|^{1-\al}}\right\|_{L^1(\Rd)}\\*[7pt]&=c\left\|\Delta\partial_if_Q*\frac{x_j}{|x|^{2-\al}}\right\|_{L^1(\Rd)}\le Cl(Q)^\al,
\end{split}
\end{equation*}
where the last integral can be estimated in an analogous way as \eqref{L1}. This finishes the proof of \eqref{L1riesz} and the lemma.
\end{proof}

For the proof of Lemma \ref{localization} we need the following preliminary lemma.

\begin{lemma}\label{prelocalization}
 Let $T$ be a compactly supported distribution in $\Rd$ with growth $\al$. Then, for each coordinate $i$, the distribution $(x_i^{2n-1} / |x|^{2n-1+\al})
* \varphi_Q T$ is an integrable function in the interior of $\frac 1 4 Q$ and 
$$
\int_{\frac 1 4 Q}\left|\left(\frac{x_i^{2n-1}}{|x|^{2n-1+\al}}*\varphi_QT\right)(y)\right|dy\leq C \, G_\al(T)\;l(Q)^2,
$$
where  $C$ is a positive constant.
\end{lemma}

For $\al=1$ the proof Lemma \ref{prelocalization} can be found in \cite{cmpt2}. In $\mathbb{R}^d$, for $n=1$ and $0<\al<d$, the proof is given in \cite{tams}.  
Although the scheme of our proof is the same as in the papers cited above, several difficulties arise due to the fact that we are considering more general kernels, namely kernels involving non-integer indexes $\al$ and $n\in\N$.

For the rest of the section we will assume, without loss of generality, that $i=1$ and we will write $K_1(x)=x_1^{2n-1}/|x|^{2n-1+\al}$.

\begin{proof}[Proof of Lemma \ref{prelocalization}]

We will 
prove that $K_1* \varphi_Q T$~is in
$L^{p}(2Q)$ for each $p$ in $1\le p<2.$ Indeed, fix any $q$ satisfying $2<q<\infty$ and call $p$ the dual exponent, so that $1<p<2$. We need to estimate the action of
$K_1 * \varphi_Q T$ on functions $\psi \in \cc^\infty_0(2Q)$ in
terms of $\|\psi\|_{q} $. We clearly have
$$
\langle K_1 * \varphi_Q T, \psi\rangle = \langle T, \varphi_Q(K_1 * \psi)\rangle.
$$
We claim that, for an appropriate
positive constant $C $, the test function
\begin{equation*}
f_Q=\frac{\varphi_Q(K_1 * \psi)}{C \,l(Q)^{\frac{2}{p}-\al} \|\psi\|_{q}}
\end{equation*}
satisfies the normalization inequalities \eqref{normalization1} in
the definition of $G_\al(T)$.  Once this is proved, by the definition of $G_\al(T)$ we get that $|\langle K_1 * \varphi_Q T, \psi\rangle | \le C\,
l(Q)^{\frac{2}{p}}\|\psi\|_{q} \,G_\al(T),$
and therefore $\|K_1 * \varphi_Q T \|_{L^{p}(2Q)} \le C\,
l(Q)^{\frac{2}{p}}G_\al(T).$ Hence
\begin{equation*}
\begin{split}
\frac{1}{|\frac{1}{4}Q|}\int_{\frac{1}{4} Q} |(K_1 * \varphi_Q
T)(x)|\,dx &\le 16\frac{1}{|Q|}\int_Q |(K_1 * \varphi_Q
T)(x)|\,dx \\*[7pt]
& \le 16\left(\frac{1}{|Q|}\int_Q |(K_1
* \varphi_Q T)(x)|^{p} \,dx\right)^{\frac{1}{p}}\\*[7pt]
& \le C\,G_\al(T),
\end{split}
\end{equation*}
which proves Lemma \ref{prelocalization}.

Notice that since $\Delta(\varphi_Q(K_1 * \psi)$ is not in $L ^1(Q),$ to prove the claim, we cannot use Lemma \ref{nicecondition}.  Therefore we have to check that, for $i=1,2,$
\begin{equation*}
\left\|\Delta\partial_if_Q*\frac1{|x|^{1-\al}}\right\|_{H^1(\Rd)}\le C\;l(Q)^\al.
\end{equation*}
This is equivalent to checking conditions
\begin{equation}\label{normaL1}
\left\|\Delta\partial_i\left(\varphi_Q(K_1*\psi)\right)*\frac1{|x|^{1-\al}}\right\|_{L^1(\Rd)}\le C\;l(Q)^{\frac2p}\|\psi\|_q
\end{equation}
and
\begin{equation}\label{normaL1riesz}
\left\|R_j(\Delta\partial_i\left(\varphi_Q(K_1*\psi)\right))*\frac1{|x|^{1-\al}}\right\|_{L^1(\Rd)}\le C\;l(Q)^{\frac2p}\|\psi\|_q
\end{equation}
for $i,j=1,2$.

By Fubini and H\"older,
\begin{equation}\label{pre}
\int_Q|(K_1*\psi)(y)|dy\le \int_{2Q}|\psi(z)|\int_Q\frac{dydz}{|z-y|^{\alpha}}\le C\|\psi\|_ql(Q)^{\frac{2}{p}+2-\al}.
\end{equation}
In the same way one can obtain 
\begin{equation}\label{pre2}
\int_Q|(\partial_iK_1*\psi)(y)|dy\le \int_{2Q}|\psi(z)|\int_Q\frac{dydz}{|z-y|^{1+\alpha}}\le C\|\psi\|_ql(Q)^{\frac{2}{p}+1-\al},\;i=1,2,.
\end{equation}
To check  \eqref{normaL1} we compute first the $L^1$-norm in $(2Q)^c$ by bringing all derivatives to the kernel $|x|^{\al-1}$, using Fubini and \eqref{pre}.  Then
\begin{equation}\label{outside}
\begin{split}
\left\|\Delta\partial_i\left(\varphi_Q(K_1*\psi)\right)*\frac1{|x|^{1-\al}}\right\|_{L^1((2Q)^c)}&\le
C\int_{Q}|(K_1*\psi)(y)|\int_{(2Q)^c}\frac{dxdy}{|y-x|^{4-\al}}\\*[7pt]&\le C\|\psi\|_ql(Q)^{\frac{2}{p}}.
\end{split}
\end{equation}
Now we are left to compute the $L^1$-norm in $2Q$ of the integral in \eqref{normaL1}. For this, we bring the Laplacian to the kernel $|x|^{\al-1}$. 
Since for $i=1,2,$ we clearly have $\partial_i\left(\varphi_Q(K_1*\psi)\right))=\partial_i\varphi_Q(K_1*\psi)+\varphi_Q\partial_i(K_1*\psi)$,  adding and substracting 
some terms to get integrability, we get

\begin{equation}\label{inside}
\begin{split}
\left\|\Delta\partial_i\left(\varphi_Q(K_1*\psi)\right)*\frac1{|x|^{1-\al}}\right\|_{L^1(2Q)} &\le
C\int_{2Q}\left|\int_Q\frac{(\varphi_Q(y)-\varphi_Q(x))(\partial_iK_1*\psi)(y)}{|y-x|^{3-\al}}dy\right|dx\\*[7pt] &+
C\int_{Q}|\varphi_Q(x)|\left|\left(\Delta\partial_i K_1*\psi*\frac1{|y|^{1-\al}}\right)(x)\right|dx\\*[7pt] &+
C\int_Q\left|\int_{Q^c}\frac{\varphi_Q(x)(\partial_iK_1*\psi)(y)}{|y-x|^{3-\al}}dy\right|dx\\*[7pt] &+ 
C\int_{2Q}\left|\int_Q\frac{(\partial_i\varphi_Q(y)-\partial_i\varphi_Q(x))(K_1*\psi)(y)}{|y-x|^{3-\al}}dy\right|dx\\*[7pt] &+
C\int_{Q}|\partial_i\varphi_Q(x)|\left|\left(\Delta K_1*\psi*\frac1{|y|^{1-\al}}\right)(x)\right|dx\\*[7pt] &+ 
C\int_Q\left|\int_{Q^c}\frac{\partial_i\varphi_Q(x)(K_1*\psi)(y)}{|y-x|^{3-\al}}dy\right|dx\\*[7pt] &= A_1+A_2+A_3+A_4+A_5+A_6,
\end{split}
\end{equation}
the last identity being a definition for $A_l$, $1\le l\le 6$.

The mean value theorem, Fubini and \eqref{pre2}, give us

$$A_1\le C l(Q)^{-1}\int_{Q}|(\partial_iK_1*\psi)(y)|\int_{2Q}\frac1{|y-x|^{2-\al}}dx\;dy\le C\|\psi\|_ql(Q)^{\frac{2}{p}}.$$
  
The same reasoning but using \eqref{pre} instead of \eqref{pre2}, give us  $A_4\le C\|\psi\|_ql(Q)^{\frac{2}{p}}.$

We deal now with term $A_2$. By Lemma \ref{fourier}, taking Fourier transform of the convolution $\Delta\partial_i K_1*\psi*\frac1{|y|^{1-\al}}$, one sees that

$$\widehat{\left(\Delta\partial_i K_1*\psi*\frac1{|y|^{1-\al}}\right)}(\xi)=c\frac{\xi_i\xi_1p(\xi_1,\xi_2)}{|\xi|^{2n}}\widehat{\psi}(\xi).$$
Therefore, since the homogeneous polynomial $\xi_i\xi_1p(\xi_1,\xi_2)$, of degree $2n$, has no non-vanishing zeros, by \cite[Theorem 4.15, p.82]{duandikoetxea}, we obtain that 
$$\left(\Delta\partial_i K_1*\psi*\frac1{|y|^{1-\al}}\right)(x)=c\psi+cS_0(\psi)(x),$$
for some constant $c$ and some smooth homogeneous Calder\'on-Zygmund operator $S_0$.

Now using H\"older's inequality and the fact that Calder\'on-Zygmund operators preserve $L^q(\Rd)$, $1<q<\infty$, we get $A_2\le Cl(Q)^{2/p}\|\psi\|_q$.

To estimate $A_3$, notice that $\varphi_Q$ is supported on $Q$, therefore

\begin{equation*}
\begin{split}
A_3&\le C \int_Q\left|\int_{3Q\setminus Q}\frac{(\varphi_Q(x)-\varphi_Q(y))(\partial_iK_1*\psi)(x)}{|y-x|^{3-\al}}dy\right| dx
\\*[7pt] & +C \int_Q|\varphi_Q(x)|\int_{(3Q)^c}\frac{\left|(\partial_iK_1*\psi)(y)\right|}{|y-x|^{3-\al}}dy\;dx=A_{31}+A_{32}.
\end{split}
\end{equation*}

For $A_{31}$ we use the mean value theorem and argue as in the estimate of $A_1$. We deal now with $A_{32}$:
\begin{equation*}
 \begin{split}
   A_{32}&\le C\int_Q\int_{(3Q)^{c}}\frac1{|y-x|^{3-\al}}\int_{2Q}\frac{|\psi(z)|}{|z-y|^{1+\al}}dz\;dy\;dx\\*[7pt]&
\le C l(Q)^{-1-\al}\|\psi\|_1\int_Q\int_{(3Q)^{c}}\frac1{|y-x|^{3-\al}}dy\;dx\\*[7pt]&\le Cl(Q)^{-1-\al}\|\psi\|_ql(Q)^{\frac 2 p}l(Q)^{1+\al}=C l(Q)^{\frac{2}{p}}\|\psi\|_q,
 \end{split}
\end{equation*} using H\"older's inequality.
To estimate terms $A_5$ and $A_6$, one argues in a similar manner, we leave the details to the reader. This finishes the proof of \eqref{normaL1}.

We are still left with checking that condition \eqref{normaL1riesz} holds. Notice that by \eqref{remark},
\begin{equation*}
 \begin{split}
  \left\|R_j(\Delta\partial_i\left(\varphi_Q(K_1*\psi)\right))*\frac1{|x|^{1-\al}}\right\|_{L^1(\Rd)}&
   =c\left\|\Delta\partial_i\left(\varphi_Q(K_1*\psi)\right)*\frac{x_j}{|x|^{2-\al}}\right\|_{L^1(\Rd)}\\*[7pt]&=B_1+B_2,
 \end{split}
\end{equation*}
where $B_1$ and $B_2$ denote the above $L^1$ norm in $(2Q)^c$ and in $2Q$ respectively. To estimate $B_1$ we transfer all derivatives to the kernel $x_j/|x|^{2-\al}$ 
and argue as in \eqref{outside}. The estimate of $B_2$ follows the same reasoning as \eqref{inside}.
\end{proof}

For the reader's convenience, we repeat the main points of the proof of the localization lemma, for more details see \cite{cmpt2}.

\begin{proof}[Proof of Lemma \ref{localization}]

Let $x\in(\frac32 Q)^c$. Since $|(K_1*\varphi_QT)(x)|=l(Q)^{-\al}|\langle T, l(Q)^\al\varphi_Q(y)K_1(x-y)\rangle|,$
by \eqref{ourgrowth} and Lemma \ref{nicecondition}, the required estimate of the $L^\infty-$norm of the function $K_1*\varphi_QT$ is equivalent to checking that 
$f_Q(y)=l(Q)^{\a}\varphi_Q(y)K_1(x-y)$ satisfies $\|\Delta f_Q\|_{L^1(Q)}\le C$, which is easily seen to hold for this case.
 
If $x\in\frac32 Q$, the boundedness of $\varphi_Q$ and $T*K_1$ implies that
$$|(K_1*\varphi_QT)(x)|\le |(K_1*\varphi_QT)(x)-\varphi_Q(x)(K_1*T)(x)|+\|\varphi_Q\|_\infty\|K_1*T\|_\infty.$$
We consider now $\psi_Q\in\cc_0^\infty(\Rd)$ such that $\psi\equiv 1$ in $2Q$, $\psi\equiv 0$ in $(4Q)^c$, $\|\psi_Q\|_\infty\le C$, 
$\|\nabla\psi_Q\|_\infty\le Cl(Q)^{-1}$ and $\|\Delta\psi_Q\|_\infty\le Cl(Q)^{-2}$. Set $K_1^x(y)=K_1(x-y)$. Then, 
\begin{equation}
\label{oldfor}
\begin{split}
 |(K_1*\varphi_QT)(x)-\varphi_Q(x)(K_1*T)(x)|&\leq|\langle T,\psi_Q(\varphi_Q-\varphi_Q(x))K_1^x\rangle|\\*[7pt]&
 +\|\varphi_Q \|_{\infty}|\langle T,(1-\psi_Q)K_1^x\rangle|=A+B.
\end{split}
\end{equation}

In fact, for the first term in the right hand side of \eqref{oldfor} to make sense, one needs to resort to a standard regularization process, whose details may be found in \cite[Lemma 12]{mpv2} for example. 

The estimate of the term $A$ is a consequence of the $\al-$growth of the distribution (see \eqref{ourgrowth}) and Lemma \ref{nicecondition}, 
because the mean value theorem implies that $f_Q=l(Q)^\al\psi_Q(\varphi_Q-\varphi_Q(x))K_1^x$ satisfies $\|\Delta f_Q\|_1\le C$.

We turn now to $B$. By Lemma \ref{prelocalization}, there exists a Lebesgue point of $K_1*\psi_QT$, $x_0\in Q$, such that $|(K_1*\psi_QT)(x_0)|\le CG_\al(T)$. 
Then $|(K_1*(1-\psi_Q)T)(x_0)|\le C(\|K_1*T\|_\infty+G_\al(T))$, which implies
$$B\le C|\langle T,(1-\psi_Q)(K_1^x-K_1^{x_0})\rangle|+C(\|K_1*T\|_\infty+G_\al(T)).$$
To estimate $|\langle
T,(1-\psi_Q)(K_1^x-K_1^{x_0})\rangle|$, we decompose
$\Rd\setminus \{x\}$ into a union of rings $$N_j=\{z\in
\Rd:2^j\,l(Q)\leq|z-x|\leq 2^{j+1}\,l(Q)\},\quad j\in\mathbb{Z},$$
and consider functions $\varphi_j$ in ${\mathcal
C}^\infty_0(\Rd)$, with support contained in
$$N^*_j=\{z\in
\Rd:2^{j-1}\,l(Q)\leq|z-x|\leq 2^{j+2}\,l(Q)\},\quad
j\in\mathbb{Z},$$ such that $\|\varphi_j\|_\infty\le C$, $\|\nabla\varphi_j\|_\infty\leq C 
\,(2^j\,l(Q))^{-1}$, $\|\Delta\varphi_j\|_{\infty}\le C\,(2^j\,l(Q))^{-2}$ and $\sum_j\varphi_j=1$ on
$\Rd\setminus\{x\}$. Since $x\in\frac 3 2 Q$  the smallest ring
$N^*_j$ that intersects $(2Q)^c$ is $N^*_{-3}$. Therefore  we have
\begin{equation*}
\begin{split}
 |\langle T,(1-\psi_Q)(K_1^{x}-K_1^{x_0})\rangle|
 &=\left|\left\langle T,\sum_{j\geq -3}\varphi_j(1-\psi_Q)(K_1^{x}-K_1^{x_0})\right\rangle\right|\\*[7pt]
 &\leq\left|\left\langle T,\sum_{j\in I}\varphi_{j}(1-\psi_Q)(K_1^{x}-K_1^{x_0})\right\rangle \right|\\*[7pt]
&\quad+\sum_{j\in J}|\langle T,\varphi_{j}(K_1^{x}-K_1^{x_0})\rangle|,
\end{split}
\end{equation*}
where $I$ denotes the set of indices $j\geq -3$ such that the
support of $\varphi_j$ intersects $4Q$  and $J$ denotes the remaining
indices, namely those $j \geq -3 $ such that $\varphi_j$ vanishes
on $4Q$. Notice that the cardinality of $I$ is bounded by a
positive constant.

Set
$$g =C\,l(Q)^\al\sum_{j\in I}\varphi_j(1-\psi_Q)\,(K_1^{x}-K_1^{x_0}),$$
and for $j\in J$
$$g_j=C\,2^j(2^{j}\,l(Q))^\al\,\varphi_j\,(K_1^{x}-K_1^{x_0}).$$
We leave it to the reader to verify that the test functions $g$ and $g_j$, $j\in J$,
satisfy the normalization inequalities \eqref{normalization1} in
the definition of $G_\al(T)$ for an appropriate choice of the (small)
constant $C$ (In fact one can check that the condition in Lemma \ref{nicecondition} holds for these functions). Once this is available, using the $\al$ growth
condition of $T$ we obtain
\begin{equation*}
\begin{split}
 |\langle T,(1-\psi_Q)(K_1^{x}-K_1^{x_0})\rangle |&\leq C l(Q)^{-\al}|\langle T,g\rangle|+ C \sum_{j\in J} 2^{-j}(2^jl(Q))^{-\al}|\langle T,g_j\rangle |\\*[7pt]
 &\leq C\,G_\al(T) + C\sum_{j\geq -3}2^{-j}\,G_\al(T)\leq C\,G_\al(T),
\end{split}
\end{equation*}
which completes the proof of Lemma \ref{localization}.
\end{proof}

\section{Relationship between the capacities $\ga$ and non linear potentials}\label{secoutline}

This section will complete the proof of Theorem \ref{main} by showing the equivalence between the capacities $\gamma_{\a,+}^n$ and $C_{\frac23(2-\a).\frac 3 2}$.

For our purposes, the description of Riesz capacities in terms of Wolff potentials is more useful than the definition of $C_{s,p}$ in \eqref{wolffcap}.  
The Wolff potential of a positive Radon measure $\mu$ is defined by
$$W^\mu_{s,p}(x)=\int_0^\infty\left(\frac{\mu(B(x,r))}{r^{2-sp}}\right)^{q-1}\frac{dr}{r},\;\;x\in\Rd,$$
The Wolff Energy of $\mu$ is $$E_{s,p}(\mu)=\int_{\Rd}W^\mu_{s,p}(x)d\mu(x).$$
A well known theorem of Wolff (see \cite{adamshedberg}, Theorem 4.5.4, p. 110) asserts that
\begin{equation}\label{wolffineq}
C^{-1}\sup_\mu\frac1{E_{s,p}(E)^{p-1}}\le C_{s,p}(E)\le C\sup_\mu\frac1{E_{s,p}(E)^{p-1}},
\end{equation}
the supremum taken over the probability measures $\mu$ supported on $E$. Here $C$ stands for a positive constant depending only on $s$, $p$ and the dimension.

To understand the relationship between the capacities $\ga$ and non linear potentials, we need to recall the 
characterization of these capacities in terms of the symmetrization method.

Let $\mu$ be a positive measure and $0<\al<1$. For $x\in\Rd$ set,
$$p_{\al,n}^2(\mu)(x)=\iint p_{\al,n}(x,y,z)d\mu(y)d\mu(z),$$
$$M_\al\mu(x)=\sup_{r>0}\frac{\mu(B(x,r))}{r^\al}$$ and
$$U_{\al,n}^\mu(x)=M_\al\mu(x)+p_{\al,n}^2(\mu)(x).$$
We denote  the energy associated to this last potential by $$\mathcal{E}_{\a,n}(\mu)=\int_{\Rd}U_{\al,n}^\mu(x)d\mu(x).$$

Notice that Corollary \ref{permalfa} states that for any $n\in\N$, given three distinct points $x,y,z\in\Rd$, $p_\a^n(x,y,z)\approx p_\a^1(x,y,z)$. Hence, for any $n\in\N$
\begin{equation}\label{comparabilitympv}
{\mathcal E}_{\a,n}(\mu)\approx \mathcal{E}_{\a,1}(\mu).
\end{equation}

Recall from \cite[Lemma 4.1]{mpv},  that for a compact set $K\subset\Rd$ and $0<\a<1$, $$\displaystyle{\gamma_{\a,+}^1(K)\approx\sup_{\mu}\frac1{\mathcal{E}_{\a,1}(\mu)}},$$ 
the supremum taken over the probability measures $\mu$ supported on $K$.
Adapting the proof of Lemma 4.1 in \cite{mpv} to our situation (using the reproduction formula from Lemma \ref{reproductionformula} and \eqref{ourgrowth}), we get that
$$\gamma_{\a,+}^n(K)\approx\sup_{\mu}\frac1{\mathcal{E}_{\a,n}(\mu)},$$ where the supremum is taken over the probability measures $\mu$ supported on $K$.

The explanation given in Step 1 of section \ref{secsketch} implies that $$\gamma_{\a}^n(K)\approx\gamma_{\a,+}^n(K),$$ hence we deduce that $$\gamma_{\a}^n(K)\approx \sup_{\mu}\frac1{\mathcal{E}_{\a,1}(\mu)}.$$
Lemma 4.2 in \cite{mpv} shows that for any positive Radon measure $\mu$, the energies $\mathcal{E}_{\a,1}(\mu)$ and $E_{\frac 2 3(2-\a),\frac 3 2}(\mu)$ are comparable.
Now, Wolff's inequality \eqref{wolffineq}, with $s=2(2-\a)/3$ and $p=3/2$, (see the proof of the main Theorem in \cite[p. 221]{mpv}) finishes the proof of Theorem \ref{main}.

\section{Rectifiability and $L^2$-boundedness of $T_n$}\label{rect}

Recalling \eqref{permdef} and \eqref{permdef2} for any Borel measure $\mu$ we define 
\begin{equation}
\label{permtriple}
p_{1,n}(\mu)=\iiint p_{1,n}(x,y,z) d \mu (x)d \mu (y) d \mu (z).
\end{equation}
The following lemma relates the finiteness of $p_{1,n}$ to the $L^2(\mu)$-boundedness of the operator $T_n$
\begin{lemma} 
\label{mmvcl2}
Let $\mu$ be a continuous positive Radon measure in $\Rn$ with linear growth. If the operator $T_n$ is bounded in $L^2(\mu)$ then there exists a constant $C$ such 
that for any ball $B$, $$\iiint_{B^3}p(x,y,z) d \mu (x)d \mu (y)d \mu (z) \leq C \diam (B).$$
\end{lemma}
The proof of Lemma \ref{mmvcl2} can be found in \cite[Lemma 2.1]{mmv}. There it is stated and proved for the Cauchy transform but the proof is identical in our case.
When $p_{1,n}(x,y,z)$ is replaced by the square of the Menger curvature $c(x,y,z)$, recall \eqref{meng} and \eqref{mel}, the triple integral in \eqref{permtriple} is called the curvature of $\mu$ and is denoted by $c^2(\mu)$. A famous theorem of David and L\'eger \cite{Leger}, which was also one of the cornerstones in the proof of Vitushkin's conjecture by David in \cite{david}, states that if $E \subset \Rn$ has finite length and $c^2(\ha^1 \lfloor E)< \infty$ then $E$ is rectifiable. Here we obtain the following generalization of the David-Leger Theorem. 
\begin{teo}
\label{legerperm}
Let $E \subset \Rn$ be a Borel set such that $0<\mathcal{H}^1(E)<\infty$ and 
$p_{1,n}(\mathcal{H}^1 \lfloor E)<\infty$, then the set $E$ is rectifiable.
\end{teo}
\begin{remarkthep} We first note that statement (1) of Theorem \ref{recthm1} follows immediately from Lemma \ref{mmvcl2} and Theorem \ref{legerperm}. Theorem \ref{legerperm} was earlier proved in \cite{cmpt} for $d=2$. We stress that the constraint $d=2$ in \cite[Theorem 1.2 (i)]{cmpt} is essentially  used in the proofs of \cite[Proposition 2.1]{cmpt}  and \cite[Lemma 2.3]{cmpt} which only go through in the plane. Nevertheless in no other instance the arguments in \cite{cmpt} depend on the ambient space being $2$-dimensional. Proposition \ref{posperm} bypasses this issue by using completely different reasoning, and generalizes  \cite[Proposition 2.1]{cmpt}  and \cite[Lemma 2.3]{cmpt} in Euclidean spaces 
of arbitrary dimension. Furthermore it removes the assumption of the triangles with comparable sides which was also essential in the proofs of \cite[Proposition 2.1]{cmpt}  
and \cite[Lemma 2.3]{cmpt}. With Proposition \ref{posperm} at our disposal we obtain (i) by following the arguments from \cite[Sections 3-7]{cmpt} 
without any changes. In several cases in \cite[Sections 3-6]{cmpt} there are references to several components from \cite{Leger} but 
this does not create any problem, since the proof in  \cite{Leger} holds for any $\Rn$.

The proof of (ii) from Theorem \ref{recthm1} follows, as in (i), by Proposition \ref{posperm} and \cite[Section 8]{cmpt}, as the arguments there do not depend on the dimension of the ambient space.
\end{remarkthep}


\appendix 
\section{Proofs of Propositions \ref{lemmaperm} and \ref{posperm}}\label{secpermap}

For simplicity we let $n$ odd. Then for $0<\a\leq 1$
\begin{equation*}
K^i_{\a,n}(x)=\frac{x_i^{n}}{|x|^{n+\a}}, \quad x =(x_1,\dots,x_d) \in \Rn \stm \{0\}.
\end{equation*}

\begin{proof}[Proof of Proposition \ref{lemmaperm}]
Write $a=y-x$ and  $b=z-y$; then $a+b=z-x$. Without loss of generality we can assume that $|a|\leq |b|\leq|a+b|$. A simple computation yields

\begin{equation}
\label{permi}
\begin{split}
&p^i_{\a,n}(x,y,z)\\
&\quad= K^i_{\a,n}(x-y)\,K^i_{\a,n}(x-z) + K^i_{\a,n}(y-x)\,K^i_{\a,n}(y-z) + K^i_{\a,n}(z-x)\,K^i_{\a,n}(z-y)\\
&\quad= K^i_{\a,n}(-a)\,K^i_{\a,n}(-a-b) + K^i_{\a,n}(a)\,K^i_{\a,n}(-b) + K^i_{\a,n}(a+b)\,K^1_{\a,n}(b)\\
&\quad=K^i_{\a,n}(a+b)K^i_{\a,n}(a)+K^i_{\a,n}(a+b)K^i_{\a,n}(b)-K^i_{\a,n}(a)K^i_{\a,n}(b)\\
&\quad=\frac{(a_i+b_i)^na_i^n|b|^{n+\a}+(a_i+b_i)^nb_i^n|a|^{n+\a}-a_i^nb_i^n|a+b|^{n+\a}}{|a|^{n+\a}|b|^{n+\a}|a+b|^{n+\a}}.
\end{split}
\end{equation}
If $a_ib_i=0$ the proof is immediate. Take for example $a_i=0$. Then we trivially obtain 
\begin{equation}
\label{aibizero}
p^i_{\a,n}(x,y,z)=\frac{b_i^{2n}}{|b|^{n+\a}|a+b|^{n+\a}} \approx \frac{ \ M_i^{2n}}{L(x,y,z)^{2\a+2n}}.
\end{equation}

To prove the upper bound inequality in \eqref{icomparability} we distinguish two cases.\newline

{\em Case} $a_ib_i>0:$ Without loss of generality assume $a_i> 0$ and $b_i> 0$. In case $a_i<0$ and $b_i<0$,
\begin{equation*}
\begin{split}
&(a_i+b_i)^na_i^n|b|^{n+\a}+(a_i+b_i)^nb_i^n|a|^{n+\a}-a_i^nb_i^n|a+b|\\
&\quad  \quad=(|a_i|+|b_i|)^n|a_i|^n|b|^{n+\a}+(|a_i|+|b_i|)^n|b_i|^n|a|^{n+\a}-|a_i|^n|b_i|^n|a+b|^{n+\a}
\end{split}
\end{equation*}
and thus it can be reduced to the case where both coordinates are positive.

Notice that since $|a|\le |b|\le |a+b|$, $a_i\le |a|$ and $0<\a<1$,
\begin{equation*}
\begin{split}
p^i_{\a,n}(x,y,z)&= \frac{(a_i+b_i)^nb_i^n}{|b|^{n+\a}|a+b|^{n+\a}}+\frac{a_i^n\left((a_i+b_i)^n|b|^{n+\a}-b_i^n|a+b|^{n+\a}\right)}{|a|^{n+\a}|b|^{n+\a}|a+b|^{n+\a}}\\
&\le\frac{1}{|b|^{\a}|a+b|^{\a}}+\frac{(a_i+b_i)^n-b_i^n}{|a|^\a|a+b|^{n-\a}}\\
&\le \frac{1}{|b|^{\a}|a+b|^{\a}}+\frac{a_i^\a}{|a|^\a}\frac{\sum_{k=1}^n\binom{n}{k}a_i^{k-\a}b_i^{n-k}}{|b|^{n+\al}}\\
&\le \frac{1}{|b|^{\a}|a+b|^{\a}}+\frac{\sum_{k=1}^n\binom{n}{k}|a|^{k-\a}|b|^{n-k}}{|b|^{n+\al}}\\&\le \frac{1}{|b|^{\a}|a+b|^{\a}}+\frac{B(n)|b|^{n-\al}}{|b|^{n+\a}}
\le \frac{B(n,\a)}{|a+b|^{2\a}},
\end{split}
\end{equation*}
where the last inequality comes from $|a+b|\le 2|b|$, which follows from the triangle inequality and the fact that $|a|\le|b|$.\newline

{\em Case} $a_i b_i <0:$ Without loss of generality we can assume that $a_i<0$, $b_i> 0$ and $b_i\le |a_i|$, the other cases follow 
analogously by interchanging the roles of $a_i$ and $b_i$. 
\begin{equation*}
\begin{split}
p^i_{\a,n}(x,y,z)&= \frac{(|a_i|-b_i)^n|a_i|^n|b|^{n+\a}-(|a_i|-b_i)^nb_i^n|a|^{n+\a}+|a_i|^nb_i^n|a+b|^{n+\a}}{|a|^{n+\a}|b|^{n+\a}|a+b|^{n+\a}}\\
&\le\frac{(|a_i|-b_i)^n|a_i|^n|b|^{n+\a}+|a_i|^nb_i^n|a+b|^{n+\a}}{|a|^{n+\a}|b|^{n+\a}|a+b|^{n+\a}}\\
&\le\frac{|a|^{2n}(|b|^{n+\a}+|a+b|^{n+\a})}{|a|^{n+\a}|b|^{n+\a}|a+b|^{n+\a}}=|a|^{n-\a}\left(\frac1{|a+b|^{n+\a}}+\frac1{|b|^{n+\a}}\right)\\
&\le\frac{2^{2\a}+1}{|a+b|^{2\a}}.
\end{split}
\end{equation*}
since $b_1\le |a_1|$, $0<|a_1|-b_1<|a_1|<|a|$ and $|a|\le|b|\le|a+b|$.

We now prove the lower bound estimate in (\ref{icomparability}).  \newline

{\em Case}  $a_ib_i > 0:$ As explained in the proof of the upper bound inequality the proof can be reduced to the case when $a_i>0$ and $b_i>0$. 
Setting $t=|b|/|a|$ in (\ref{permi}) and noticing that $|a+b|/|a| \leq 1+t$ we get
\begin{equation}
\label{pfbound}
\pia(x,y,z) \geq \frac{a_i^n(a_i+b_i)^nt^{n+\a}-b_i^n a_i^n (1+t)^{n+\a}+b_i^n (a_i+b_i)^n }{|b|^{n+\a}|a+b|^{n+\a}}:=\frac{f_1(t)}{|b|^{n+\a}|a+b|^{n+\a}}.
\end{equation}
Then it readily follows that the unique zero of
$$f'_1(t)=a_i^n(n+\a)(t^{n+\a-1} (a_i+b_i)^n-b_i^n (1+t)^{n+\a-1})$$
is
$$t^\ast=\frac{1}{\left( \frac{a_i}{b_i}+1 \right)^{\frac{n}{n+\a-1}}-1}>0.$$
Moreover $f_1$ attains its minimum at $t^\ast$ because $f_1'(0)=-(n+\a)a_i^n b_i^n$ and 
$\lim_{t \ra \infty} f_1'(t)=\lim_{t \ra \infty} ((b_i+a_i)^n-b_i^n)t^{n+\a-1}=+\infty$.

We first consider the case when $t^\ast >1$. Then we deduce that 
\begin{equation}\label{s}
0<\frac{a_i}{b_i}<2^{\frac{n + \a-1}{n}}-1<1,
\end{equation}
 the last inequality coming from $\a<1$. 
Therefore $a_i<b_i$. Setting $s=a_i/b_i$ we obtain 
\begin{equation*}
f_1(t)
=b_i^{2n} \left(s^n(1+s)^nt^{n+\a}-s^n (1+t)^{n+\a}+(s+1)^n\right)
\end{equation*}
and it follows easily that
$$f_1(t^\ast) =b_i^{2n} (1+s)^n \left(1-\frac{s^n}{((s+1)^{\frac{n}{n+\a-1}}-1)^{n+\a-1}} \right).$$
A direct computation shows that the function 
$$g_1(s)=1-\frac{s^n}{((s+1)^{\frac{n}{n+\a-1}}-1)^{n+\a-1}}$$ is decreasing. Then, by \eqref{s}, $g_1$ attains its minimum at $s=2^{\frac{n + \a-1}{n}}-1$. Therefore 
\begin{equation}
\label{tastg1}
f_1(t) \geq f_1(t^\ast) \geq b_i^{2n} (1- (2^{\frac{n + \a-1}{n}}-1)^n):=b_i^{2n}A_1(n,\a)
\end{equation}
and since $\a<1$, $A_1(n, \a) >0.$ 

We now consider the case when $t^\ast \leq 1$ and notice that as $t \geq 1$ we have $f_1(t) \geq f_1(1)$. As before for $s=\min \{a_i, b_i\} /\max \{a_i, b_i\}$
\begin{equation}
\label{ftast2}
\begin{split}
f_1(t) \geq f_1(1)&= a_i^n (a_i+b_i)^n-a_i^n b_i^n 2^{n+\a}+b_i^n (a_i+b_i)^n \\
&=(a_i+b_i)^n (\max \{a_i,b_i\})^n \left(s^n-\left( \frac{s}{s+1} \right)^n 2^{n+\a}+1 \right) \\
&:=(a_i+b_i)^n (\max \{a_i,b_i\})^n g_2(s).
\end{split}
\end{equation}
It follows easily that the only non-zero root of
$$g_2'(s)=n s^{n-1}\left(1-\frac{2^{n+\a}}{(1+s)^{n+1}}\right)$$
is $$s^\ast = 2^{\frac{n+\a}{n+1}}-1.$$ Since $\a \in (0,1)$, then $2^{\frac{n+\a-1}{n}}-1<s^\ast<1.$ Furthermore notice that 
$g_2'(2^{\frac{n+\a-1}{n}}-1)<0$ and $g_2'(1)>0$. Hence $g_2$ attains its minimum at $s^\ast$. Therefore
\begin{equation}\label{tsmall1}
\begin{split}
g_2(s)\geq g_2(s^\ast)&=(2^{\frac{n+\a}{n+1}}-1)^n \left( 1-\frac{2^{n+\a}}{2^{\frac{(n+\a)n}{n+1}}} \right)+1 \\
&=1-(2^{\frac{n+\a}{n+1}}-1)^{n+1}:=A_2(n,\a)>0,
\end{split}
\end{equation}
the positivity of the constant $A_2(n,\a)$ coming from inequality $\a<1$. Therefore (\ref{pfbound}) together with \eqref{tastg1}, (\ref{ftast2}) and \eqref{tsmall1} imply that
\begin{equation}
\label{aibipos2}
\begin{split}
\pia (x,y,z) \geq A(n,\a) \frac{M_i^{2n}}{L(x,y,z)^{2n+2a}},
\end{split}
\end{equation}
for some positive constant $A(n,\a)$. Hence we have finished the proof when $a_i \, b_i > 0$.

{\em Case} $a_i \, b_i < 0:$ 
Setting  $t=|b| /|a|$ and using (\ref{pfbound}) we get that
\begin{equation}
\label{secondf1}
\begin{split}
 \pia (x,y,z)&=\frac{1}{|b|^{n+\a}|a+b|^{n+\a}}(a_i^n(a_i+b_i)^nt^{n+\a}-b_i^n a_i^n (1+t)^{n+\a}+b_i^n (a_i+b_i)^n ) \\
 & \geq \frac{1}{|b|^{n+\a}|a+b|^{n+\a}}(a_i^n(a_i+b_i)^nt^{n+\a}-b_i^n a_i^n t^{n+\a}+b_i^n (a_i+b_i)^n )\\
&:= \frac{f_2(t)}{|b|^{n+\a}|a+b|^{n+\a}}.
\end{split}
\end{equation}
Notice that $f_2$ is an increasing function because $a_i^2+a_ib_i>a_ib_i$ and $n$ is odd:
\begin{equation*}
\begin{split}f'_2(t)&=(n+\a)t^{n+\a-1}(a_i^n(a_i+b_i)^n-a_i^nb_i^n)=(n+\a)t^{n+\a-1}((a_i^2+a_ib_i)^n-a_i^nb_i^n) > 0.
\end{split}
\end{equation*}
Therefore since $t \geq 1$ we have that 
$$f_2(t) \geq f_2(1)=(a_i^n+b_i^n)(a_i+b_i)^n-b_i^n a_i^n.$$
We assume that $|b_i| \geq |a_i|$, the case where $|b_i| <|a_i|$ can be treated in the exact same manner. We first consider the case where $a_i>0$ and $b_i<0$.
Let
$$h(r)=(r-|b_i|)^n(r^n-|b_i|^n)+r^n|b_i|^n.$$
Then 
$$h'(r)=n\left((r-|b_i|)^{n-1}(r^n-|b_i|^n)+r^{n-1}((r-|b_i|)^n+|b_i|^n)\right).$$
Notice that $$h'(|b_i|/2)=0.$$ Furthermore $$h'(r)>0 \mbox{ for }|b_i|/2 <r\leq |b_i|.$$ To see this notice that since $0<|b_i|-r<r$, then
$(|b_i|-r)^{n-1} <r^{n-1}$.Therefore since $r^n-|b_i|^n<0$
\begin{equation*}
\begin{split}
h'(r)>n(r^{n-1}(r^n-|b_i|^n)+r^{n-1}((r-|b_i|)^n+|b_i|^n))=nr^{n-1}(r^n-(|b_i|-r)^n)>0.
\end{split}
\end{equation*}
With an identical argument one sees that $h'(r)<0$ for $0<r \leq |b_i|/2$. Hence it follows that, for $0<r\leq |b_i|$,
$$h(r) \geq h(|b_i|/2) \geq \frac{|b_i|^{2n}}{2^n}.$$
Since $a_i \in (0, |b_i|]$ we get that $f_2(1) \geq \frac{|b_i|^{2n}}{2^n}$ and by (\ref{secondf1})
\begin{equation}
\label{aibineg1}
 \pia (x,y,z) \geq 2^{-n} \frac{b_i^{2n}}{|b|^{n+\a}|a+b|^{n+\a}} \geq A_3(n) \frac{M_i^{2n}}{L(x,y,z)^{2n+2\a}}.
\end{equation}

The case where $a_i <0$ and $b_i>0$ is very similar. In this case instead of the function $h$ we consider the function
$l(r)=(r+b_i)^n(r^n+b_i^n)-r^nb_i^n$
for $-|b_i|/2\leq x <0$ and we show that in that range, 
$$l(r) \geq l(-|b_i|/2) \geq b_i^{2n}/2^n.$$
Therefore as $a_i \in [-|b_i|,0)$, $f_2(1) \geq \frac{|b_i|^{2n}}{2^n}$ and we obtain from \eqref{secondf1}
\begin{equation}
\label{aibineg2}
 \pia (x,y,z) \geq 2^{-n} \frac{b_i^{2n}}{|b|^{n+\a}|a+b|^{n+\a}} \geq A_3(n) \frac{M_i^{2n}}{L(x,y,z)^{2n+2\a}}.
\end{equation}
Therefore the proof of the lower bound follows by (\ref{aibipos2}), (\ref{aibineg1}) and (\ref{aibineg2}).

\end{proof}

\begin{remark1}
\label{aibinegp1}
Notice that in the proof of the lower bound inequality when $a_ib_i <0$, we do not make use of the fact that $\a<1$. Therefore (\ref{aibineg1}) and (\ref{aibineg2}) remain valid in the case where $\a=1$.
\end{remark1}

\begin{proof}[Proof of Proposition \ref{posperm}] For simplicity we let $p^i_{1,n}:=p^i_n$ for $i=1,\dots,d$. Let $a=y-x, b=z-y$ then $a+b=z-x$ and without loss of generality 
we can assume that $|a|\leq |b|\leq|a+b|=1$. In case $x_i=y_i=z_i$, then trivially by (\ref{permi}), $p^i_n(x,y,z)=0$. Hence we can assume that 
$a_i \neq 0$ or $b_i \neq 0$ and, by (\ref{permi}), for $\a=1$ , assuming without loss of generality that $b_i \neq 0$, we get 
\begin{equation}
\begin{split}
\label{perm1}
p^i_{n}(x,y,z)= \frac{(a_i+b_i)^n b_i^n \left(\left(\frac{a_i}{b_i} \right)^n |b|^{n+1}+|a|^{n+1}-\frac{a_i^n}{(a_i+b_i)^n}\right)}{|a|^{n+1}|b|^{n+1}}.
\end{split}
\end{equation}

If the points $x,y,z$ are collinear then the initial assumption $|a|\leq|b|\leq|a+b|$ implies that $|a|+|b|=|a+b|$. Furthermore $b=\lambda a$ for some $\l\neq 0$. 
We provide the details in the case when $\lambda >0$ as the remaining case is identical. We have by (\ref{perm1})
\begin{equation*}
\begin{split}
p^i_{n}(x,y,z)&=\frac{(a_i+\lambda a_i)^n \lambda^n a_i^n \left(\left(\frac{1}{\lambda} \right)^n \lambda^{n+1} |a|^{n+1}+|a|^{n+1}-\left(\frac{1}{1+\l}\right)^n\right)}{|a|^{n+1}|b|^{n+1}} \\
&=\frac{a_i^{2n} \lambda^n \left(\left( (1+\l)|a| \right)^{n+1}-1\right)}{|a|^{n+1}|b|^{n+1}}\\
&=\frac{a_i^{2n} \lambda^n }{|a|^{n+1}|b|^{n+1}}\left((1+\l)|a|-1\right)\sum_{j=0}^{n}((1+\l)|a|)^j=0
\end{split}
\end{equation*}
because $(1+\l)|a|-1=|a|+|b|-1=0$. 

We will now turn our attention to the case when the points $x,y,z$ are not collinear. We will consider several cases. \newline

\textit{Case} $a_ib_i>0.$ As in the proof of Proposition \ref{icomparability} we only have to consider the case when $a_i,b_i>0$. We first consider the subcase $0<|a|\leq |b|<|a+b|=1$.

Setting $w=a_i/b_i$ in \eqref{perm1} we get
\begin{equation*}
p^i_{n}(x,y,z)=\frac{(a_i+b_i)^n b_i^n}{|a|^{n+1}|b|^{n+1}}  f(w)
\end{equation*}
with $$f(w)=w^n|b|^{n+1}+|a|^{n+1}-\left(1+\frac{1}{w}\right)^{-n}.$$
Notice that the only non-vanishing admissible root of the equation
$$f'(w)=nw^{n-1}\left(|b|^{n+1}-\left(\frac{1}{w+1}\right)^{n+1} \right)=0$$
is $w=|b|^{-1}-1$. Furthermore it follows easily that 
$$\lim_{w\ra 0^+}f(w)=|a|^{n+1}>0 \quad \text{and}\quad \lim_{w \ra +\infty}f(w)=+\infty$$
hence $f:(0,\infty) \ra \R$ attains its minimum at $|b|^{-1}-1$. After a direct computation we get that
$$f(|b|^{-1}-1)=|a|^{n+1}-(1-|b|)^{n+1}.$$
We can now write,
\begin{equation*}
\begin{split}
|a|^{n+1}&-(1-|b|)^{n+1}=|a|^{n+1}\left(1-\left( \frac{1-|b|}{|a|}\right)^{n+1} \right)\\
&=|a|^{n+1}\left(1- \frac{1-|b|}{|a|}\right) \sum_{j=0}^n \left(\frac{1-|b|}{|a|} \right)^j=|a|^n(|a|-1+|b|) \sum_{j=0}^n \left(\frac{1-|b|}{|a|} \right)^j. \\
\end{split}
\end{equation*}
Therefore 
\begin{equation}
\label{estal1}p^i_{n}(x,y,z)\geq \frac{(a_i+b_i)^n b_i^n}{|a|^{n+1}|b|^{n+1}}|a|^n (|a|-1+|b|)\sum_{j=0}^n \left(\frac{1-|b|}{|a|} \right)^j.
\end{equation}

Recall that, by Heron's formula, the area of the triangle determined by $x,y,z \in \R^d$ is given by 
$$\text{area}(T_{x,y,z})=\frac{1}{2} \sqrt{|a+b|^2|a|^2-\left(\frac{|a+b|^2+|a|^2-|b|^2}{2}\right)^2},$$
where $a=y-x, b=z-y$ and $a+b=z-x$.
Hence
$$16 \, \text{area}(T_{x,y,z})^2=(2|a+b||a|-(|a+b|^2+|a|^2-|b|^2))(2|a+b||a|+|a+b|^2+|a|^2-|b|^2).$$
Plugging this identity into Menger's curvature formula we get
\begin{equation*}
\begin{split}
c^2(x,y,z)&=\frac{16 \, \text{area}(T_{x,y,z})^2}{|a|^2|b|^2|a+b|^2}\\
&=\frac{(2|a+b||b|-|a+b|^2-|a|^2+|b|^2)(2|a+b||b|+|a+b|^2+|a|^2-|b|^2)}{|a|^2|b|^2|a+b|^2}\\
&=\frac{(|b|^2-(|a+b|-|a|)^2)((|a|+|a+b|)^2-|b|^2)}{|a|^2|b|^2|a+b|^2}\\
&=\frac{(|b|-|a+b|+|a|)(|b|+|a+b|-|a|)(|a|+|a+b|-|b|)(|b|+|a+b|+|a|)}{|a|^2|b|^2|a+b|^2}
\end{split}
\end{equation*}
and since we are assuming $|a+b|=1$,
\begin{equation}
\label{herrocurv}
c^2(x,y,z)=\frac{(|b|+|a|-1)(|b|+1-|a|)(|a|+1-|b|)(|b|+1+|a|)}{|a|^2|b|^2}.
\end{equation}
By (\ref{estal1}) and (\ref{herrocurv}) we get that
\begin{equation*}
\begin{split} 
p^i_{n}(x,y,z)&\geq \frac{(a_i+b_i)^n b_i^n}{|a|^{n+1}|b|^{n+1}}\frac{|a|^n|a|^2|b|^2}{(|b|+1-|a|)(|a|+1-|b|)(|b|+1+|a|)}
\sum_{j=0}^n \left(\frac{1-|b|}{|a|} \right)^j c^2(x,y,z)\\
&\ge\frac{b_i^{2n}|a||b|}{|b|^n(|b|+1-|a|)(|a|+1-|b|)(|b|+1+|a|)}c^2(x,y,z),
\end{split}
\end{equation*}
the last inequality coming from $a_i+b_i\ge b_i$ and the fact that the sum above is greater than one. Using the triangle inequality, $1=|a+b|\le|a|+|b|$, and the fact 
that $1=|a+b|\le 2|b|$, we obtain
\begin{equation*}
 p^i_{n}(x,y,z)\ge \frac{b_i^{2n}}{12|b|^n}c^2(x,y,z)\ge c(n)\frac{ b_i^{2n}}{|b|^{2n}}c^2(x,y,z).
\end{equation*}

To complete the proof in case $a_ib_i>0$, we are left with the situation $|b|=|a+b|=1$.  By (\ref{perm1})
\begin{equation*}
p^i_{n}(x,y,z)=\frac{(a_i+b_i)^n b_i^n \left(\left(\frac{a_i}{b_i} \right)^n +|a|^{n+\a}-\left(\frac{a_i}{a_i+b_i}\right)^n\right)}{|a|^{n+\a}}\geq (a_i+b_i)^n b_i^n \geq b_i^{2n},
\end{equation*}
because $\frac{a_i}{b_i}>\frac{a_i}{a_i+b_i}$ and thus $\left(\frac{a_i}{b_i} \right)^n>\left(\frac{a_i}{a_i+b_i}\right)^n.$
Hence
$$p^i_{n}(x,y,z) \geq \frac{b_i^{2n}}{|b|^{2n}}|b|^{-2} \gtrsim \frac{b_i^{2n}}{|b|^{2n}} c^2(x,y,z).$$

\textit{Case} $a_i b_i<0$. It follows from Remark \ref{aibinegp1} (see \eqref{aibineg2} with $\a=1$).

\textit{Case}  $a_i \,b_i=0$. Since we have assumed that $b\neq 0$ we have that $a_i=0$ and by \eqref{aibizero}, with $\a=1$, we are done.

Therefore we have shown that whenever $x,y,z$ are not collinear and they do not lie in the hyperplane $x_i=y_i=z_i$, then $p^i_{n}(x,y,z)>0$. This finishes the proof of (i).
Furthermore, we have shown that if this is the case, then 
\begin{equation}
\label{p11}
p^i_{n}(x,y,z) \geq C(n) \frac{ b_i^{2n}}{|b|^{2n}}c^2(x,y,z).
\end{equation}
\newline
For the proof of (ii) notice that since $\mang (V_j, L_{y,z}) \geq \theta_0$ there exists some coordinate $i_0 \neq j$ such that 
$$|b_{i_0}|=|y_{i_0}-z_{i_0}|\geq C(\theta_0)|y-z|=C(\theta_0)|b|,$$ 
hence (ii) follows by (\ref{p11}).
\end{proof}

\textbf{Acknowledgement}. We would like to thank Joan Mateu and Xavier Tolsa for valuable conversations during the preparation of this paper.

\end{document}